\documentclass[twoside, 12pt]{amsart}
\usepackage{amssymb}

\usepackage{amsfonts}
\usepackage{amsthm}
\usepackage[english]{babel}

\usepackage[T1]{fontenc}

\usepackage{color}

\newtheorem{theorem}{Theorem}[section]
\newtheorem{lemma}[theorem]{Lemma}
\newtheorem{proposition}[theorem]{Proposition}

\theoremstyle{definition}
\newtheorem{remark}[theorem]{Remark}
\newtheorem{example}[theorem]{Example}

\begin{document}

\author{Jacek
Jakubowski and
 Maciej Wi\'sniewolski }

\title[Verhulst process]{Exact distribution of Verhulst process}

\maketitle

\begin{center}
{\small
 Institute of Mathematics, University
of Warsaw \\
  Banacha 2, 02-097 Warszawa, Poland \\
e-mail: {\tt jakub@mimuw.edu.pl } \\
and \\
  {\tt wisniewolski@mimuw.edu.pl } }
\end{center}

\begin{abstract}
We investigate a Verhulst process, which is the special functional of geometric Brownian motion and   has many applications, among others in biology and  in stochastic volatility models. We present
an exact form of density of a one dimensional distribution of Verhulst process. Simple formula for the density of Verhulst process is obtained in the special case, when the drift of geometric Brownian motion is equal to $-\frac12$. Some special properties of this process are discussed, e.g. it turns out that  under Girsanov's change of measure  a Verhulst process still remains a  Verhulst process but with different parameters.
\end{abstract}

\noindent
\begin{quote}


\end{quote}

\section{Introduction}
 The paper is devoted to study the distribution of Verhulst process. A Verhulst process is the special functional of geometric Brownian motion with drift and its integral. The process had been studied among others  by  Mackevi$\check{c}$ius \cite{Mac} and Lungu and {\O}ksendal \cite{LO}, where it is called a process of population growth in stochastic crowded environment.  The results
in both papers were obtained by a some kind of approximations. In {\O}ksendal \cite{O}
there were indicated  some applications of a Verhulst process in finance and biology.
Mackevi$\check{c}$ius \cite{Mac1}  during his conference presentation on Vilnius Conference 2014 indicates on the need to find the closed-form expressions for Laplace transform and exact distribution of Verhulst process.
In this paper we provide the formulas for one dimensional distribution of the process.
Results obtained by Yor and collaborators  in several papers and monographes (see, e.g. \cite{CMY}, \cite{DYM}, \cite{DMY}, \cite{Mans08}, \cite{Mat-I}, \cite{MatII}, \cite{MatIII}) on the distribution of
$(B_t^{(\mu)},\int_0^te^{B_u^{(\mu)}}du)$, where $B_t^{(\mu)}=B_t + \mu t$ is a Brownian motion with drift, give us the background in providing closed formulas for the density of Verhulst process, which in case of drift $\mu = -\frac{1}{2}$ becomes especially simple.
We present also some interesting and important  properties of the Verhulst process among them the fact that  a Verhulst process remains,  under Girsanov's change of measure, a Verhulst proces but with different parameters.
The ideas presented below are original and are not easy consequences of the previous results.

\section{Distribution and properties of Verhulst process}
We work on  a complete probability
space $(\Omega,\mathcal{F},\mathbb{P})$ with filtration $(\mathcal{F}_t)_{t \in
[0,\infty)}$ and Brownian motion $B$ defined on it.
We define a Verhulst process $\theta^{(\mu,\beta)}$ starting from $1$ as the functional of a Brownian motion with drift
   \begin{equation}\label{VP}
	\theta_t^{(\mu,\beta)} = \frac{e^{B_t +\mu t}}{1+\beta\int_0^te^{B_s+\mu s}ds}.
\end{equation}
It is easy to see that this functional is
  the unique strong solution of SDE
\begin{equation}\label{SDE}
	d\theta_t^{(\mu,\beta)} = \theta_t^{(\mu,\beta)}dB_t + \Big((\mu+1/2)\theta_t^{(\mu,\beta)}-\beta(\theta_t^{(\mu,\beta)})^2\Big)dt, \ \theta_0^{(\mu,\beta)} = 1,
\end{equation}
for $\mu\in\mathbb{R},$ $\beta\geq 0$ and $ t\geq 0$.
It is worth to note that taking $\theta_t^{(\mu,\beta)}$ starting  from any $x> 0$  does not change any of its probabilistic properties (see also Section 3). As it was mentioned,  the process $\theta^{(\mu,\beta)}$ is also known in literature as a process of population growth in crowded stochastic environment process (see \cite{O}).
The approximation of the process was presented in  \cite{Mac}. The properties of  Verhulst process were also studied by  Jakubowski and Wi\'sniewolski \cite{JWII}. In particular it was shown there (see Theorem 2.19) that for $\lambda >0, a\in\mathbb{R}$, we have
\begin{equation}\label{Lap}
    \mathbb{E}\Big[ \exp\Big(-\frac{\lambda e^{aB^{(\mu)}_t}}{1+\beta
    A_t^{(\mu)}}\Big) \Big]
    = \mathbb{E} \Big[F_{B_t^{(\mu)}}\Big(\beta^{-1}R^{(\lambda e^{aB^{(\mu)}_t})}(1/2)\Big)\Big],
\end{equation}
where $B^{(\mu)}_t = B_t +\mu t$, $ A_t^{(\mu)} = \int_0^te^{2B^{(\mu)}_s}ds$, and  $F_x$ is given by
\begin{equation}\label{postac-F}
    F_{x}(z) =   \exp\Big(-\frac{\varphi_z(x)^2 - x^2}{2t}\Big)
\end{equation}
with
\begin{align}
\label{varphi}
   & \varphi_x(y)  = \mbox{arcosh}(xe^{-y} + \cosh(y))\\
                            & = \ln\Big(xe^{-y} + \cosh(y)+\sqrt{x^2e^{-2y}+\sinh^2(y) +2xe^{-y}\cosh(y)}\Big).\notag
   \end{align}
Moreover, $R^x$ is a squared Bessel process with the index $-1$ starting at $x$ and independent of
 $(B^{(\mu)}_t,t\geq 0)$.

Observe that from \eqref{Lap}
we easily obtain the Laplace transform of $\theta_t^{(\mu,\beta)}$.
\begin{proposition} For $\lambda > 0$ we have
\begin{equation}
	\mathbb{E}e^{-\lambda \theta_t^{(\mu,\beta)}} = \mathbb{E} \Big[F_{B_t^{(2\mu)}}\Big((4\beta)^{-1}R^{(\lambda e^{2B^{(2\mu)}_t})}(1/2)\Big)\Big].
\end{equation}
\end{proposition}
\begin{proof} Exactly in the same manner as in the proof of Theorem 2.20 in \cite{JWII}, we have
\begin{align*}
\theta_{4t}^{(\mu,\beta)} =
\frac{e^{2(B_{4t}/2 +2t\mu)}}{1+4\beta\int_0^{t}e^{2(B_{4t}/2 +2t\mu)}du}.
\end{align*}
Since $B_{4t}/2$ is a Brownian motion,  the statement  follows from
(\ref{Lap}) with
$ a = 2$, $\mu$ replaced by $2\mu$ and $\beta$ replaced by $4\beta$.
\end{proof}
The following lemma will be used further.
\begin{lemma}\label{LQ} Fix $\gamma >0$, $\mu\in\mathbb{R},$ $\beta\geq 0$. Then
\begin{equation}\label{defM}
	M^{(\gamma,\mu,\beta)}_t = e^{-\gamma\int_0^t\theta_s^{(\mu,\beta)}dB_s -\frac{\gamma^2}{2}\int_0^t(\theta_s^{(\mu,\beta)})^2ds}, \qquad t \in[0,\infty)
\end{equation}
is a martingale.
\end{lemma}
\begin{proof} It is enough to prove that for fixed $T>0$ the local martingale $M^{(\gamma,\mu,\beta)}_t,  t \in[0,T]$ is bounded. It is so, since SDE \eqref{SDE} implies that
\begin{align*}
&\exp \Big({-\gamma\int_0^t\theta_s^{(\mu,\beta)}dB_s -\frac{\gamma^2}{2}\int_0^t(\theta_s^{(\mu,\beta)})^2ds}\Big) \\ &= e^{-\gamma\Big(\theta_t^{(\mu,\beta)}-1 -(\mu+1/2)\int_0^t\theta_s^{(\mu,\beta)}ds+\beta\int_0^t(\theta_s^{(\mu,\beta)})^2ds\Big) -\frac{\gamma^2}{2}\int_0^t(\theta_s^{(\mu,\beta)})^2ds}\\
%&= e^{-\gamma(\theta_t^{(\mu,\beta)}-1) %+\gamma(\mu+1/2)\int_0^t\theta_s^{(\mu,\beta)}ds-(\gamma\beta+\frac{\gamma^2}{2})\int_0^t(\theta_s^{(\mu,\beta)})^2ds}\\
&= e^{-\gamma(\theta_t^{(\mu,\beta)}-1)-(\gamma\beta+\frac{\gamma^2}{2})\int_0^t\Big((\theta_s^{(\mu,\beta)})^2- 2\theta_s^{(\mu,\beta)}\frac{\gamma(\mu+1/2)}{\gamma^2+2\gamma\beta} +
 \Big(\frac{\gamma(\mu+1/2)}{\gamma^2+2\gamma\beta}\Big)^2\Big)ds } \times \\
& \times e^{ \frac{(\gamma(\mu+1/2))^2} {4(\gamma\beta+\gamma^2/2)}t } <\infty.
\end{align*}
\end{proof}
\begin{remark}
One can wonder if $	\overline{M}^{(\gamma,\mu,\beta)}_t = e^{\gamma\int_0^t\theta_s^{(\mu,\beta)}dB_s -\frac{\gamma^2}{2}\int_0^t(\theta_s^{(\mu,\beta)})^2ds}$, for a fixed $\gamma >0$, could be a martingale as well. In Remark 1.1 \cite{DMY} it was noticed that $	\overline{M}^{(\gamma,\mu,\beta)}$
can not be a martingale.
\end{remark}
Lemma \ref{LQ} enables us to define a new probability measure
\begin{equation}\label{MQ}
	\frac{d\mathbb{Q}^{(\gamma,\mu,\beta,T)}}{d\mathbb{P}}\Big|_{\mathcal{F}_T} = 	M^{(\gamma,\mu,\beta)}_T .
\end{equation}
From Girsanov's theorem $V_s = B_s + \gamma\int_0^s\theta_u^{(\mu,\beta)}du$,  $s\leq T$, is a Brownian motion under $\mathbb{Q}^{(\gamma,\mu,\beta)}$. This leads us to the following
\begin{theorem} Let $\theta_t^{(\mu,\beta)}$ be defined by (\ref{VP}) and $\mathbb{Q}^{(\gamma,\mu,\beta,T)}$ by (\ref{MQ}). Then
\begin{equation}
	\theta_t^{(\mu,\beta)} = \frac{e^{V_t +\mu t}}{1+ (\beta + \gamma)\int_0^te^{V_u +\mu u }du}, \qquad t\leq T,
\end{equation}
where $V$ is a standard Brownian motion under $\mathbb{Q}^{(\gamma,\mu,\beta,T)}$.
\end{theorem}
\begin{proof} From Ito's lemma and (\ref{SDE}) follows that
\begin{align}\label{Int}
	\ln \theta_t^{(\mu,\beta)} = B_t + \int_0^t(\mu - \beta \theta_s^{(\mu,\beta)})ds.
\end{align}
Taking $V_t = B_t + \gamma\int_0^t\theta_u^{(\mu,\beta)}du$, a Brownian motion under $\mathbb{Q}^{(\gamma,\mu,\beta,T)}$, we obtain from the last equality that
\begin{align}\label{EV}
	V_t + \mu t = \ln \theta_t^{(\mu,\beta)} + (\beta + \gamma)\int_0^t\theta_u^{(\mu,\beta)}du.
\end{align}
Thus direct computation yields
\begin{align*}
	\int_0^te^{V_u + \mu u}du &= \int_0^t\theta_u^{(\mu,\beta)}e^{(\beta + \gamma)\int_0^u\theta_s^{(\mu,\beta)}ds}du\\
	&= \frac{1}{\gamma+\beta}\Big(e^{(\beta + \gamma)\int_0^t\theta_u^{(\mu,\beta)}du} -1\Big).
\end{align*}
From the last equality and (\ref{EV}) we finally have
\begin{align*}
	 \theta_t^{(\mu,\beta)} = \frac{e^{V_t +\mu t}}{1+ (\beta + \gamma)\int_0^te^{V_u +\mu u }du}.
\end{align*}
\end{proof}
This theorem justifies that a  Verhulst process $\theta^{(\mu,\beta)}$ on $(\Omega,\mathcal{F},\mathbb{P})$  remains a Verhulst process under  $\mathbb{Q}^{(\gamma,\mu,\beta,T)}$, though with different parameters.
\begin{proposition} \label{Stb}
If  $(\theta_t^{(\mu,\beta)})$, $t \leq T$,  is a Verhulst process with parameters $(\mu,\beta)$ on $(\Omega,\mathcal{F},\mathbb{P})$, then  $(\theta_t^{(\mu,\beta)})$  is   on $(\Omega,\mathcal{F},\mathbb{Q}^{(\gamma,\mu,\beta,T)})$ a Verhulst process with parameters $(\mu,\gamma + \beta)$.
\end{proposition}
We are ready to derive the formula for the density of one dimensional distribution of Verhulst process, i.e. the density of $\theta_t^{(\mu,\beta)}$ for a fixed $t> 0$.
Let us introduce the following notation
\begin{align*}
	\mathbb{P}(\theta_t^{(\mu,\beta)}\in dx) = g_t(\beta,x)dx.
\end{align*}
Observe that for $\beta = 0$, $g_t(0,x)$ is  the density of geometric Brownian motion $e^{B_t +\mu t}$.
Proposition \ref{Stb} implies that
\begin{align*}
	g_t(\beta+\gamma,x) = g_t(\beta, x)\mathbb{E}(M^{(\gamma,\mu,\beta)}_t|\theta_t^{(\mu,\beta)} = x).
\end{align*}
In particular, for $\beta = 0$ we obtain
\begin{align}\label{Dens}
	g_t(\gamma,x) = g_t(0, x)\mathbb{E}(M^{(\gamma,\mu,0)}_t|e^{B_t +\mu t} = x).
\end{align}
Theorem 8.1 in Matsumoto - Yor \cite{MatII} states that: \\
for any $t>0, \lambda > 0, v > 0, x\in\mathbb{R}$
\begin{align}\label{MYOR}
	&\psi_t^{(\mu)}(v,x)\mathbb{E}\Big(e^{-\frac{\lambda^2}{2}\int_0^te^{2B_u +2\mu u}du} \Big|\int_0^te^{B_u+\mu u}du = v, B_t + \mu t = x\Big)\\
	&= e^{\mu x - \mu^2t/2}\frac{\lambda}{4\sinh(\lambda v/2)}e^{-\lambda(1+e^x)\coth(\lambda v/2)}\Theta(\phi(v,x,\lambda), t/4),\notag
\end{align}
where
\begin{align*}
	\psi_t^{(\mu)}(v,x) &= \frac{1}{16}e^{\mu x - \mu^2t/2}\frac{1}{v}e^{-\frac{2(1+e^x)}{v}}\Theta(4e^{x/2}/v, t/4),\\
	\Theta(r,t) &= \frac{r}{\sqrt{2\pi^3t}}e^{\frac{\pi^2}{2t}}\int_0^{\infty}e^{-\frac{z^2}{2t}}e^{-r\cos(z)}\sinh(z)\sin(\pi z/t)dz, \\
	\phi(v,x,\lambda) &= \frac{2\lambda e^{x/2}}{\sinh(\lambda v/2)}.
\end{align*}
Using this result we can write the formula for density of Verhulst process.
\begin{theorem}	\label{th:Vdens} The density of Verhulst process is given by
\begin{align*}
		g_t(\gamma,x) = g_t(0, x)e^{-\gamma(x-1)}\mathbb{E}H_t(a_t^{(\mu)},x),
\end{align*}
where $a_t^{(\mu)} = \int_0^te^{B_u+\mu u}du$, and
\begin{align*}
	H_t(y,x) = e^{\gamma(\mu+1/2)y}\mathbb{E}\Big(e^{-\frac{\gamma^2}{2}\int_0^te^{2B_u +2\mu u}du}|a_t^{(\mu)} = y, B_t + \mu t = \ln x\Big).
\end{align*}
\end{theorem}
\begin{proof} We have from (\ref{defM}) and (\ref{SDE})
\begin{align*}
\mathbb{E}(M^{(\gamma,\mu,0)}_t|e^{B_t +\mu t} = x) &= e^{-\gamma(x-1)}\mathbb{E}\Big(e^{\gamma(\mu +1/2)a_t^{(\mu)} - \frac{\gamma^2}{2}\int_0^te^{2B_u +2\mu u}du}|e^{B_t +\mu t} = x\Big)\\
&= e^{-\gamma(x-1)}\int_0^{\infty}H_t(y,x)\mathbb{P}(a_t^{(\mu)}\in dy)\\
&= e^{-\gamma(x-1)}\mathbb{E}H_t(a_t^{(\mu)},x).
\end{align*}
Thus, the result follows from (\ref{Dens}).
\end{proof}
\begin{remark}
Taking together Theorem \ref{th:Vdens} and (\ref{MYOR}) enable  us to write down the close formula for density of Verhulst process. The density of $a_t^{(\mu)}$ can be obtained from the formula
\begin{align*}
	\mathbb{P}(a_t^{(\mu)}\in dv, B_t +\mu t\in dx) = \psi_t^{(\mu)}(v,x)dvdx
\end{align*}
(see \cite[Theorem 4.1]{MatII}).
Another  formula for density of  $a_t^{(\mu)}$ can be also found on page 612 in \cite{SB}, formula (1.8.4).
\end{remark}
\begin{proposition}
For a Verhulst process $\theta_t^{(\mu,\beta)}$, a Brownian motion $B$, $\beta>0$ and $\mu \neq -1/2$ we have
\begin{align*}
	\theta_t^{(\mu,\beta)}e^{\beta\int_0^t\theta_s^{(\mu,\beta)}ds} &= e^{B_t + \mu t},\\
	\mathbb{E}e^{\beta\int_0^t\theta_s^{(\mu,\beta)}ds} &= 1 + \frac{\beta}{\mu + 1/2}\Big(e^{(\mu+1/2)t}-1\Big).
\end{align*}
\end{proposition}
\begin{proof} The first equality follows immediately from (\ref{Int}). For the second observe that
\begin{align} \label{1/6}
	e^{\beta\int_0^t\theta_s^{(\mu,\beta)}ds} = (1+\beta\int_0^te^{B_u +\mu u}du),
\end{align}
so
\begin{align*}
	e^{(\mu+1/2)t} = \mathbb{E}e^{B_t+\mu t} = \mathbb{E}\theta_t^{(\mu,\beta)}e^{\beta\int_0^t\theta_s^{(\mu,\beta)}ds}
= \frac{1}{\beta}\frac{\partial}{\partial t}\mathbb{E}e^{\beta\int_0^t\theta_s^{(\mu,\beta)}ds}.
\end{align*}
Thus
\begin{align*}
\mathbb{E}e^{\beta\int_0^t\theta_s^{(\mu,\beta)}ds} = 1 + \frac{\beta}{\mu + 1/2}\Big(e^{(\mu+1/2)t}-1\Big).
\end{align*}
\end{proof}
As an example of applications we present a solution to problem of finding a special representation of Brownian motion with drift.
For another distributional equations of this kind see e.g. Section 13 in \cite{Mat}.
\begin{example} Let $B$ be a Brownian motion under $\mathbb{P}$, $\mu\in\mathbb{R}$ and $\alpha\in(0,1)$. Our aim is to find a measure $\mathbb{Q}$, a Brownian motion $V$ under $\mathbb{Q}$ and a random variable $G$ such that distribution of $G$ under $\mathbb{P}$ and under $\mathbb{Q}$ belongs to the same class, and for fixed $t > 0$
\begin{align} \label{2/6}
	B_t + \mu t = \alpha (V_t + \mu t) + (1-\alpha) G.	
\end{align}
Fix $T$, $T>t$, and $\gamma > 0$. To find a representation \eqref{2/6} we take $\mathbb{Q}=\mathbb{Q}^{(\gamma, \mu,\beta,T)}$ given by (\ref{MQ}) with $\beta = \frac{\gamma\alpha}{1-\alpha}$. Then $V_t = B_t + \gamma\int_0^t\theta_s^{(\mu,\beta)}ds$ is a Brownian motion under $\mathbb{Q}$. By (\ref{Int}) and \eqref{1/6}
we have
\begin{align*}
	e^{V_t +\mu t} = \theta_t^{(\mu,\beta)}\Big(1+\beta\int_0^te^{B_s +\mu s}ds\Big)^{\frac{\beta+\gamma}{\beta}}.
\end{align*}
From the definition of $\theta_t^{(\mu,\beta)}$ and the last equality we obtain
\begin{align} \label{2/7}
	e^{B_t + \mu t} = e^{\frac{\beta}{\beta+\gamma}(V_t + \mu t)}(\theta_t^{(\mu,\beta)})^{\frac{\gamma}{\beta +\gamma}}.
\end{align}
Since $\alpha= \frac{\beta}{\beta+\gamma}$, by taking ln of both sides of \eqref{2/7} we have
$$
B_t+ \mu t = \alpha (V_t + \mu t) + (1-\alpha ) \ln \theta_t^{(\mu,\beta)},
$$
which is \eqref{2/6} with $G = \ln \theta_t^{(\mu,\beta)}$.
Observe that $\ln\theta_t^{(\mu,\beta)}$ under both $\mathbb{P}$ and $\mathbb{Q}$ belongs to the same family of $\ln$ of Verhulst process. This finishes the proof.
\end{example}
Now, we present another formula for Laplace transform of Verhulst process.
\begin{proposition} \label{LapT2} Let $\theta_t^{(\mu,\beta)}$ be a Verhulst process with parameters $(\mu,\beta)$. Then for $\lambda \geq 0$
\begin{equation}
	\mathbb{E} e^{-\lambda \theta_t^{(\mu,\beta)}} = e^{\beta}\mathbb{E}e^{-(\beta+\lambda)e^{B_t+\mu t}+\beta(\mu+1/2)\int_0^te^{B_s+\mu s}ds -\frac{\beta^2}{2}\int_0^te^{2B_s+2\mu s}ds},
\end{equation}
where $B$ is a standard Brownian motion under $\mathbb{P}$.
\end{proposition}
\begin{proof}
 From Proposition \ref{Stb} we know that a geometric Brownian motion $e^{B_t+\mu t}$
 under $\mathbb{P}$ becomes a Verhulst process $\tilde{\theta}_t^{(\mu,\beta)}$ with parameters $(\mu,\beta)$ under $\mathbb{Q} = \mathbb{Q}^{(\beta,\mu,0,t)}$ given by \eqref{MQ}.  Thus
 \begin{align*}
 \mathbb{E} e^{-\lambda \theta_t^{(\mu,\beta)}} =	 \mathbb{E}_{\mathbb{Q}}e^{-\lambda \tilde{\theta}_t^{(\mu,\beta)}} = \mathbb{E} e^{-\lambda e^{B_t+\mu t}} M^{(\beta,\mu,0)}_t,
 \end{align*}
 where $M^{(\beta,\mu,0)}$ is defined by (\ref{defM}). The assertion follows from $(\ref{SDE})$ and some simple algebra.
\end{proof}
\section{Exponential random time and drift $\mu = -\frac{1}{2}$}
In this section we will consider a Verhulst process $\theta$   starting from $x>0$ with $\mu = -\frac{1}{2}$ and $\beta=x$. Thus,  $\theta$ is  of the special form
\begin{align} \label{1/7}
	\theta_t = \frac{xe^{B_t - \frac{t}{2}}}{1+x\int_0^te^{B_u - \frac{u}{2}}du},
\end{align}
where  $B$ is a Brownian motion. Let $T_{\lambda}$ be an exponential random variable with parameter $\lambda>0$, independent of $B$. We have
\begin{lemma}\label{expdens} Let $v = \sqrt{2\lambda+1/4}$. The density of $\theta_{T_{\lambda}}$ is given on $(0,\infty)$ by
\begin{equation}
	\mathbb{P}(\theta_{T_{\lambda}}\in dz) = 2\lambda e^{x-z}\sqrt{x/z^3}F_{v}(x,z)dz,
\end{equation}
where $$F_{v}(x,y) = I_{v}(x\vee y)K_{v}(x\wedge y)$$ is a product of two modified Bessel functions.
\end{lemma}
\begin{proof} For $r\geq 0$ from Proposition \ref{LapT2} (where we put $\mu = -1/2, \lambda = xr, \gamma = x$) we obtain
\begin{align*}
	\mathbb{E}e^{-r\theta_t} = e^x\mathbb{E}e^{-x(r+1)e^{B_t-t/2} - \frac{x^2}{2}\int_0^te^{2B_u -u}du}.
\end{align*}
Thus, using  \cite[Theorem 4.11]{MatII} describing the joint density
of the vector $(e^{B_{T_{\lambda}}-{T_{\lambda}}/2},\int_0^{T_{\lambda}}e^{2B_u -u}du)$, we have
\begin{align} \label{1/8}
	\mathbb{E}e^{-r\theta_{T_{\lambda}}} &= e^x\mathbb{E}e^{-x(r+1)e^{B_{T_{\lambda}}-{T_{\lambda}}/2} - \frac{x^2}{2}\int_0^{T_{\lambda}}e^{2B_u -u}du}\\ \notag
	&= e^x\int_0^{\infty}\int_0^{\infty}e^{-x(r+1)y-\frac{x^2}{2}w}\frac{\lambda}{y^{v+5/2}}p^{(v)}(w,1,y)dydw,
\end{align}
where $p^{(v)}$ is the transition probability density of the Bessel process with index $v$. Again from \cite{MatII} (see Remark 2.1) we have for $\alpha >0$
$$
	\int_0^{\infty}e^{-\alpha t}p^{(v)}(t,x,y)dt = 2y\Big(\frac{y}{x}\Big)^vF_{v}(\sqrt{2\alpha}x, \sqrt{2\alpha}y).
$$
Thus by \eqref{1/8}, Fubbini's theorem and the last identity we have
\begin{align*}
e^x&\int_0^{\infty}\int_0^{\infty}e^{-x(r+1)y-\frac{x^2}{2}w}\frac{\lambda}{y^{v+5/2}}p^{(v)}(w,1,y)dydw\\
&= \int_0^{\infty}e^{x-x(r+1)y}\frac{\lambda}{y^{v+5/2}}2y^{v+1}F_v(x,xy)dy\\
&= \int_0^{\infty}e^{x-rw -w}\frac{2\lambda}{w^{3/2}}\sqrt{x}F_v(x,w)dw
\end{align*}
and the assertion follows from the standard Laplace transform argument.
\end{proof}
Now we are ready to derive the exact formula for the density of $\theta_t$.
\begin{theorem} Fix $t> 0$. Then on $(0,\infty)$
\begin{equation}
	\mathbb{P}(\theta_{t}\in dw) = e^{-\frac{t}{8}+x -w}\sqrt{\frac{x}{w^3}}\Big(\int_0^{\infty}\frac{1}{z}e^{-\frac{z}{2}-\frac{x^2+w^2}{2z}}\Theta(xw/z,t)dz\Big)dw,
\end{equation}
where
\begin{equation}
	\Theta(r,t) = \frac{r}{\sqrt{2\pi^3t}}e^{\frac{\pi^2}{2t}}\int_0^{\infty}e^{-\frac{z^2}{2t}}e^{-r\cos(z)}\sinh(z)\sin(\pi z/t)dz.
\end{equation}
\end{theorem}
\begin{proof}
 We have on $(0,\infty)$
\begin{align}\label{RHS}
	\mathbb{P}(\theta_{T_{\lambda}}\in dw) = \lambda\int_0^{\infty}e^{-\lambda t}\mathbb{P}(\theta_{t}\in dw)dt.\,
\end{align}
where $T_{\lambda}$ is  an exponential random variable with parameter $\lambda$ independent of the process $\theta$. On the other side, from Lemma \ref{expdens} for $\lambda >0$ and $v=\sqrt{2\lambda+1/4}$, we have on $(0,\infty)$
\begin{align*}
	\mathbb{P}&(\theta_{T_{\lambda}}\in dw) = 2\lambda e^{x-w}\sqrt{\frac{x}{w^3}}F_{v}(x,w)dw\\
	&= 2\lambda e^{x-w}\sqrt{\frac{x}{w^3}}I_{v}(x\vee w)K_{v}(x\wedge w)dw\\
	&= 2\lambda e^{x-w}\sqrt{\frac{x}{w^3}}\frac12\int_0^{\infty}e^{-\frac{z}{2}-\frac{(x\vee w)^2+(x\wedge w)^2}{2z}}I_v((x\vee w)(x\wedge w)/z)\frac{dz}{z}dw\\
	&= \lambda e^{x-w}\sqrt{\frac{x}{w^3}}\int_0^{\infty}e^{-\frac{z}{2}-\frac{x^2+w^2}{2z}}I_v(xw/z)\frac{dz}{z}dw,
\end{align*}
where in the third equality we used the identity
$$
I_{v}(x)K_{v}(w)	= \frac12\int_0^{\infty}e^{-\frac{z}{2}-\frac{x^2+w^2}{2z}}I_v(xw/z)\frac{dz}{z}
$$
for $w\geq x>0$ (see (2.7) in \cite{MatII}). To continue we use another identity for modified Bessel functions (see (2.10) in \cite{MatII})
\begin{align*}
	I_v(r) = \int_0^{\infty}e^{-\frac{v^2}{2}t}\Theta(r,t)dt, \ r>0
\end{align*}
and obtain
\begin{align*}
 \lambda &e^{x-w}\sqrt{\frac{x}{w^3}}\int_0^{\infty}e^{-\frac{z}{2}-\frac{x^2+w^2}{2z}}I_v(xw/z)\frac{dz}{z}\\
 &= \lambda e^{x-w}\sqrt{\frac{x}{w^3}}\int_0^{\infty}\frac{1}{z}e^{-\frac{z}{2}-\frac{x^2+w^2}{2z}}\int_0^{\infty}e^{-\frac{v^2}{2}t}\Theta(xw/z,t)dt\\
 &= \lambda e^{x-w}\sqrt{\frac{x}{w^3}}\int_0^{\infty}\int_0^{\infty}e^{-\lambda t -\frac{t}{8}}\Theta(xw/z,t)\frac{1}{z}e^{-\frac{z}{2}-\frac{x^2+w^2}{2z}}dzdt,
\end{align*}
where in the last equality we used Fubini's theorem and that $v^2 = 2\lambda + 1/4$. To finish the proof we compare the last expression with (\ref{RHS}).
\end{proof}
Now we consider a family of processes defined by \eqref{1/7} for all $x>0$, and to underline the dependence on $x$ we denote this process by $\theta(x)$, i.e.
\begin{align*}
	\theta_t(x) = \frac{xe^{B_t - \frac{t}{2}}}{1+x\int_0^te^{B_u - \frac{u}{2}}du}.
\end{align*}
From Lemma \ref{expdens} we can deduce
\begin{proposition} Let $T_{\lambda}, \hat{T}_2$ be two independent exponential random variables, which are independent of the Brownian motion $B$. Then on $(0,\infty)$
\begin{equation}
	z^2\mathbb{P}\Big(\theta_{T_{\lambda}}(\hat{T}_2)\in dz\Big) = \mathbb{P}(\hat{T}_2\in dz)\mathbb{E}(\theta_{T_{\lambda}}(z))^2.
\end{equation}
\end{proposition}
\begin{proof}
From Lemma \ref{expdens} for $x>0$, we have on $(0,\infty)$
\begin{align*}
		e^{-2x}z^2\mathbb{P}(\theta_{T_{\lambda}}(x)\in dz) = 2\lambda e^{-x-z}\sqrt{xz}F_{v}(x,z)dz.
\end{align*}
Thus
\begin{align*}
		e^{-2x}\mathbb{E}(\theta_{T_{\lambda}}(x))^2 = \int_0^{\infty}2\lambda e^{-x-z}\sqrt{xz}F_{v}(x,z)dz.
\end{align*}
From symmetry of $F_v$, after integrating on $x$, we obtain
\begin{align*}
	z^2\int_0^{\infty}e^{-2x}\mathbb{P}(\theta_{T_{\lambda}}(x)\in dz)dx &= \Big(\int_0^{\infty}2\lambda e^{-x-z}\sqrt{xz}F_{v}(x,z)dx\Big)dz\\
	&= e^{-2z}\mathbb{E}(\theta_{T_{\lambda}}(z))^2dz.
\end{align*}
The assertion follows.
\end{proof}
\bibliographystyle{plain}

\end{document}